\documentclass[12pt,psamsfonts]{amsart}
\usepackage{amsmath}
\usepackage{amsthm}
\usepackage{amssymb}
\usepackage{amscd}
\usepackage{amsfonts}
\usepackage{amsbsy}
\usepackage[latin1]{inputenc}
\usepackage{epsfig,afterpage}
\usepackage[dvips]{psfrag}
\usepackage{psfrag}
\usepackage[all]{xy}
\usepackage{color}
\usepackage{multirow}
\usepackage{pdflscape}
\usepackage{slashbox} 
\usepackage{pict2e} 
 \usepackage{array}
\newcommand{\R}{\ensuremath{\mathbb{R}}}
\newcommand{\C}{\ensuremath{\mathbb{C}}}

\newcommand{\e}{\varepsilon}

\newcommand{\sgn}{\mbox{\normalfont sgn}}

\usepackage{overpic}

\newtheorem {theorem} {Theorem}

\newtheorem* {lemma}  {Fundamental Lemma}

\newtheorem {remark} {Remark}

\AtBeginDocument{%
  
}

\begin{document}

\title[Bifurcation of limit cycles from a non-smooth cylinder]{Bifurcation of limit cycles from a non-smooth perturbation of a two-dimensional isochronous cylinder}

\author[Buzzi, C.A., Euz\'{e}bio, R.D. and Mereu, A.C.]
{Claudio A. Buzzi$^{1}$, Rodrigo D. Euz\'{e}bio$^{1}$ and Ana C. Mereu$^2$}

\address{$^1$ Department of Mathematics,
IBILCE - UNESP Univ Estadual Paulista, Rua Cristov\~{a}o Colombo 2265, Jardim Nazareth, CEP
15.054--000, Sao Jos\'e de Rio Preto, SP, Brazil}

\address{$^2$ Department of Physics, Chemistry and Mathematics.
 UFSCar. 18052-780, Sorocaba, SP, Brazil}

\subjclass[2010]{}


\date{}
\dedicatory{}

\begin{abstract}
Detect the birth of limit cycles in non-smooth vector fields is a very important matter into the recent theory of dy\-na\-mi\-cal systems and applied sciences. The goal of this paper is to study the bifurcation of limit cycles from a continuum of periodic orbits filling up a two-dimensional isochronous cylinder of a vector field in $\mathbb{R}^{3}$. The approach involves the regularization process of non-smooth vector fields and a method based in the Malkin's bifurcation function for $C^{0}$ perturbations. The results provide sufficient conditions in order to obtain limit cycles emerging from the cylinder through smooth and non-smooth perturbations of it. To the best of our knowledge they also illustrate the implementation by the first time of a new method based in the Malkin's bifurcation function. In addition, some points concerning the number of limit cycles bifurcating from non-smooth perturbations compared with smooth ones are studied. In summary the results yield a better knowledge about limit cycles in non-smooth vector fields in $\mathbb{R}^{3}$ and explicit a manner to obtain them by performing non-smooth perturbations in codimension one Euclidean manifolds.
\end{abstract}

\maketitle


\bigskip

\section{Introduction}\label{secao introducao}

\subsection{Setting the problem}\label{secao colocacao do problema}

Non-smooth vector fields have become
certainly one of the common frontiers between Ma\-the\-matics and
Phy\-sics or Engineering.
Many authors have contributed to the study of non-smooth vector fields  (see
for instance the pioneering work \cite{Fi} or the didactic works \cite{diBernardo-livro,Marco-enciclopedia}, and references therein about details of
these multi-valued vector fields).  In our approach Filippov's convention is considered, see \cite{Fi}. So, the vector field of the model is non-smooth across a \textit{switching manifold} and it is possible for its trajectories to be
confined onto the switching manifold itself. The occurrence of such
behavior, known as \textit{sliding motion}, has been reported in a
wide range of applications. We can find important examples in electrical circuits having switches, in mechanical devices in which components collide into each other, in problems with friction, sliding or squealing, among others (see \cite{diBernardo-livro}).

\smallskip

This work concerns with the existence of limit cycles emerging from a continuum of periodic solutions filling up a two dimensional cylinder via a non-smooth perturbation. Such kind of problems are closed related to the weakest version of the famous 16th Hilbert's problem proposed by Arnol'd (see \cite{ARNOLD1} and \cite{ARNOLD2}). Arnol'd asked about the number of limit cycles bifurcating from the perturbation of a center and up to now many authors have contributed with this subject. However, the problems of perturbation of a submanifold filled up by periodic solutions which appears in the literature are usually restricted to the plane. In our opinion the perturbation of other kind of  two-dimensional manifolds has been poorly treated in the literature, and this is the goal of this paper.

\smallskip

Recently in \cite{LT} the authors investigated the problem of perturbation of a two-dimensional cylinder filled up by periodic solutions in $\mathbb{R}^{3}$ by a smooth function. In their paper, the authors illustrated the implementation of a method based in the averaging theory for computing the limit cycles bifurcating from a continuum of periodic solutions occupying a cylinder. Other papers with similar approaches can be found in \cite{LPT1} and \cite{LPT2}.

\smallskip

In this paper the goal is to generalize the study presented in \cite{LT} for a biggest class of cylinders and also take into account non-smooth perturbations. We stress out that this is not the situation considered in paper \cite{LT}. We consider the differential system
\begin{equation}\label{system_no_perturbation}
\begin{array}{l}
\dot{x}=-y+x(x^2+y^2-1),\vspace{0.2cm}\\
\dot{y}=x+y(x^2+y^2-1),\vspace{0.2cm}\\
\dot{z}=h(x,y).
\end{array}
\end{equation}

Observe that once the function $h(x,y)$ does not depend on $z$, the cylinder $\textsl{C}=\{(x,y,z)\in \R^3:x^2+y^2 =1\}$ is an invariant set for system \eqref{system_no_perturbation}. The solution passing through the point $(\cos \theta_0, \sin \theta_0,z_0)\in \textsl{C}$ at time $t=0$  is $x(t)= \cos(t+\theta_0)$, $ y(t)=\sin(t+\theta_0)$ and
\begin{equation}\label{solution}
 z(t)=z_0+\displaystyle\int_{0}^{t}h(\cos(s+\theta_{0}),\sin(s+\theta_{0}))ds.
\end{equation}

Consequently the solutions on the cylinder $\textsl{C}$ are periodic if the last integral is periodic. In order to verify such property about this integral, we must impose some conditions on the function $h$. Otherwise, the cylinder is invariant but not filled up with periodic orbits. Indeed, we will consider the functions $h$ which can be written into the form $h(x,y)=\rho(x^2+y^2)\overline{h}(x,y)$, where $\overline{h}(x,y)=\textstyle\sum_{i+j\geq 1}a_{ij}x^{i}y^{j}$. Then we will achieve conditions on the natural values $i$ and $j$ for which
\begin{equation}\label{integral}
\displaystyle\int_{0}^{t}h(\cos s,\sin s)ds=\rho(1)\cdot\displaystyle\sum_{i+j\geq 1}a_{ij}\displaystyle\int_{0}^{t}\cos^{i}s\,\sin^{j}s \,ds,
\end{equation}
is periodic, when now we take $\theta_{0}=0$ in order to simplify the expressions. The expression into the integral takes the following form
$$
\cos^{i}s\,\sin^{j}s=\displaystyle\sum_{m=0}^{\left[\frac{i+j}{2}\right]}c_{m}\cos((i+j-2m)s),
$$
or
$$
\cos^{i} s\,\sin^{j} s=\displaystyle\sum_{m=0}^{\left[\frac{i+j}{2}\right]}d_{m}\sin((i+j-2m)s),
$$
if $i+j$ is even or odd, respectively (see  \cite{BLMT}). Using the formulae below and a table of integrals one can see that in both cases the integral are periodic unless $i+j-2m=0$ when $j$ is even. Indeed, in such case the cosine of the first expression provide a constant term which is not periodic after integration. However, the condition $i+j-2m=0$ when $j$ is even implies that $i$ is also even. Then in order to live the last integral of equality \eqref{integral} periodic we must impose that $i$ and $j$ can not be even simultaneously. Moreover, it is not difficult to see that $\overline{h}(x,y)$ can be put into the following form
$$
\overline{h}(x,y)=h_{1}(x^{2},y^{2})+x\,h_{2}(x^{2},y^{2})+xy\,h_{3}(x^{2},y^{2})+y\,h_{4}(x^{2},y^{2}).
$$

Hence, since the power of $x$ and $y$ can not be even simultaneously, we are interested in the class of functions presenting the form $\tilde{h}(x,y)=x\,\phi(x^{2},y^{2})+xy\,\chi(x^{2},y^{2})+y\,\psi(x^{2},y^{2})$. Therefore, since the periodic orbits live on the cylinder $\textsl{C}$, we will take into account that the functions $h(x,y)=\rho(x^2+ y^2)\tilde{h}(x,y)$ satisfying the condition $\rho(r^2\cos^2\theta+r^2\sin^2\theta)=\rho(1)$ for $r=1$ in polar coordinates.


\smallskip

In this paper we perform a non-smooth perturbation in system \eqref{system_no_perturbation}. It means that we consider two special perturbations of system \eqref{system_no_perturbation} depending on the region of $\mathbb{R}^{3}$, which lead us to a non-smooth system. The results are obtained by using the Malkin's bifurcation function (see \cite{BLM}) after the performing of a regularization of such non-smooth system. We stress out that apart from the results presented in this paper, it has an especial importance because we exhibit a thoroughly implementation of the method presented in \cite{BLM}. As far as the authors know there is no other examples of implementation of this method in the literature.

\smallskip

In what follows, in Subsection \ref{methods}, we present the methods and tools that will be used in this paper. In Subsection \ref{preliminares} we introduce the objects that we will needed in order to state the results. Next, in Subsection \ref{main results} we state the results. In Subsection \ref{provas} we prove the results and later, in Subsection \ref{exemplos}, we present a particular example and briefly discuss some differences between performing smooth and non-smooth perturbation in system \eqref{system_no_perturbation}.

\subsection{Introducing the tools and methods}\label{methods}

Let $V$ be an arbitrarily small neighborhood of $0\in\R^2$ and consider a codimension one manifold $\Sigma$ of $\R^2$ given by
$\Sigma =f^{-1}(0),$ where $f:V\rightarrow \R$ is a smooth function
having $0\in \R$ as a regular value (i.e. $\nabla f(p)\neq 0$, for
any $p\in f^{-1}({0}))$. We call $\Sigma$ the \textit{switching
manifold} that is the separating boundary of the regions
$\Sigma^+=\{q\in V \, | \, f(q) \geq 0\}$ and $\Sigma^-=\{q \in V \,
| \, f(q)\leq 0\}$. Observe that we can assume, locally around the origin of $\R^2$, that
$f(x,y)=y$. Moreover, designate by $\chi$ the space of C$^r$-vector fields on
$V\subset\R^2$, with $r \geq 1$
large enough for our purposes. Call $\Omega$ the space of vector
fields $Z: V \rightarrow \R ^{2}$ such that
\begin{equation}\label{eq Z}
 Z(x,y)=\left\{\begin{array}{l} X(x,y),\quad $for$ \quad (x,y) \in
\Sigma^+,\\ Y(x,y),\quad $for$ \quad (x,y) \in \Sigma^-,
\end{array}\right.
\end{equation}
where $X=(X_1,X_2) , Y = (Y_1,Y_2) \in \chi$.  The trajectories of
$Z$ are solutions of  ${\dot q}=Z(q)$ and we accept it to be
multi-valued at points of $\Sigma$. The basic results of
differential equations in this context were stated by Filippov in
\cite{Fi}, that we summarize next. Indeed, consider Lie derivatives \[X.f(p)=\left\langle \nabla f(p),
X(p)\right\rangle \;\; \mbox{ and } \;\; X^i.f(p)=\left\langle
\nabla X^{i-1}. f(p), X(p)\right\rangle, i\geq 2
\]
where $\langle . , . \rangle$ is the usual inner product in $\R^2$.

\smallskip

We  distinguish the following regions on the discontinuity set
$\Sigma$:
\begin{itemize}
\item [(i)]$\Sigma^c\subseteq\Sigma$ is the \textit{sewing region} if
$(X.f)(Y.f)>0$ on $\Sigma^c$ .
\item [(ii)]$\Sigma^e\subseteq\Sigma$ is the \textit{escaping region} if
$(X.f)>0$ and $(Y.f)<0$ on $\Sigma^e$.
\item [(iii)]$\Sigma^s\subseteq\Sigma$ is the \textit{sliding region} if
$(X.f)<0$ and $(Y.f)>0$ on $\Sigma^s$.
\end{itemize}

In this paper we consider a plane separating the cylinder $\textsl{C}$ into two parts in order to perturb each one into two different functions. Nevertheless, due to the arrangement of $\textsl{C}$, we will take $y=0$ as this plane in such way that each orbit on the cylinder intersects $\Sigma$ transversally in two distinct points. It is clear that the switching manifold in this case is given by $\Sigma=F^{-1}({0})$ where $F(x,y,z)=y$. Note that the intersection of the cylinder $\textsl{C}$ with $\Sigma$ are the straight lines $x=\pm 1$; we note also that $\Sigma$ separates $\textsl{C}$ in two connected components (see Figure \ref{cilindro}).

\begin{figure}[!h]
\begin{center}
\psfrag{S}{$\Sigma$}\psfrag{C}{$\textsl{C}$}\psfrag{x}{$x$}\psfrag{y}{$y$}\psfrag{z}{$z$}\psfrag{a}{$-1$}\psfrag{b}{$1$}\psfrag{c}{$1$}
 \epsfxsize=7cm
\epsfbox{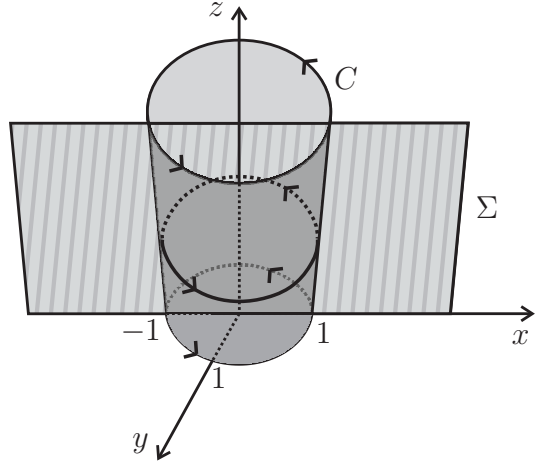} \caption{\small{The intersection between $\Sigma$ and $\textsl{C}$.}} \label{cilindro}\end{center}
\end{figure}

\smallskip

Now we perturb system \eqref{system_no_perturbation}. Taking into account the geometry of $\Sigma$, we will consider the polynomials $g^{\pm}=(p^{\pm},q^{\pm},r^{\pm})$ given by
\begin{equation}\label{perturbations}
\begin{array}{rcl}
p^{\pm}(x,y,z)&=&\displaystyle\sum_{i+j+k\leq m} a^{\pm}_{ijk}x^{i}y^{j}z^{k},\vspace{0.2cm}\\
q^{\pm}(x,y,z)&=&\displaystyle\sum_{i+j+k\leq n} b^{\pm}_{ijk}x^{i}y^{j}z^{k},\vspace{0.2cm}\\
r^{\pm}(x,y,z)&=&\displaystyle\sum_{i+j+k\leq p} c^{\pm}_{ijk}x^{i}y^{j}z^{k},
\end{array}
\end{equation}
with $i,j,k\in\mathbb{N}$ and $a_{ijk},b_{ijk},c_{ijk}\in\mathbb{R}$, $\forall i,j,k\in\mathbb{N}$. Moreover, consider the function
$$
g(x,y,z)=\dfrac{1}{2}(g^{+}(x,y,z)+g^{-}(x,y,z))+\dfrac{\sgn(y)}{2}(g^{+}(x,y,z)-g^{-}(x,y,z)),
$$
and observe that the expression of the function $g$ take different forms different depending on the signal of $y$, i.e., $g(p)=g^{+}(p)$ if $p\in\Sigma^{+}=\{y\geq 0\}$ and $g(p)=g^{-}(p)$ if $p\in\Sigma^{-}=\{y\leq 0\}$ for each $p\in\mathbb{R}^{3}$. Then, by performing a perturbation in system \eqref{system_no_perturbation} through the non-smooth function $g$ we obtain the non-smooth differential system
\begin{equation}\label{system_noncontinuous}
\dot{X_{\varepsilon}}=f(t,X)+\varepsilon g(X).
\end{equation}
where $f(t,X)$ is the vector field of system \eqref{system_no_perturbation}, $X=(x,y,z)$ and $\varepsilon$ is a small parameter.

\smallskip

Following the Filippov's convention, we have $XF(x,y,z)$ $=YF(x,$ $y,z)=y(x^2+y^2-1)+x$ when $\varepsilon=0$ and then $XF(\pm 1,0,z)\cdot$ $YF(\pm 1,0,$ $z)=1$
Therefore $\textsl{C}\cap\Sigma\subset\Sigma^{c}$. Also, if $|\varepsilon|\neq0$ is sufficiently small, the intersection $\textsl{C}\cap\Sigma$ still occurs in sewing points since the transversality of the solutions passing through sewing points is stable.

\smallskip

A powerful tool for study the perturbation of a continuum of periodic solutions as system \eqref{system_noncontinuous} is the averaging theory. Despite, in \cite{LNT} the authors exhibits a result based in the averaging theory where it is possible to consider non-smooth vector fields into the {\it standard form}, i.e., when $f(t,X)\equiv 0$. However, system \eqref{system_noncontinuous} is not in the standard form, then we can not apply the results of \cite{LNT}. In fact, once function $g$ in system \eqref{system_noncontinuous} is non-smooth, as far as the authors know there is no perturbation method in the literature that works out in this system. Nevertheless, in those cases where $g$ is $C^{0}$, we can apply a result based in the Malkin's bifurcation function presented in \cite{BLM}, even if the considered system is not in the standard form. This method is summarized in the following theorem.

\begin{theorem}\label{teotool}
Consider the $T$-periodic differential system

\begin{equation}\label{perturbado}
\dot{x}=f(t,x)+\e g(t,x,\e),
\end{equation}
where $f\in C^2(\R\times\R^n,\R^n)$ and $g\in C^0(\R\times\R^n\times[0,1],\R^n)$ are $T$-periodic in the first variable and $g$ is locally uniformly Lipschitz with respect to its second variable. For $z\in \R^n$ denote by $x(\cdot,z,\e)$ the solution of \eqref{perturbado} such that $x(0,z,\e)=z$. Assume that the unperturbed system

\begin{equation}\label{noperturbado}
\dot{x}=f(t,x)
\end{equation}
satisfies the following conditions.
\begin{itemize}
\item[i)] There exist an open ball $U\subset \R^k$ with $k\leq n$ and a function $\xi \in C^1(\overline{U},\R^n)$ such that for $h\in \overline{U}$ the $n\times k$ matrix $D\xi(z)$ has rank $k$ and $\xi(z)$ is the initial condition of a $T$-periodic solution of \eqref{noperturbado}.

\item[ii)] For each $h\in \overline{U}$ the linear system

\begin{equation}\label{linearizado}
\dot{y}=D_x f(t,x(t,z,0))y
\end{equation}
with $z=\xi(z)$ has the Floquet multiplier $+1$ with the geometric multiplicity equal to $k$.
\end{itemize}

Let $u_1(\cdot,z)$, ..., $u_k(\cdot,z)$ be linearly independent $T$-periodic solutions of the adjoint linear system

\begin{equation}\label{adjoint}
\dot{u}=-(D_x f(t,x(t,\xi(z),0)))^*u,
\end{equation}
such that $u_1(0,z)$, ..., $u_k(0,z)$ are $C^1$ with respect to $h$ and define the function $M:\overline{U}\rightarrow \R^k$ (called the Malkin's bifurcation function) by
$$M(z)=\displaystyle\int_0^T\left(\begin{array}{l}\left\langle u_1(s,z)\right.,  \left. g (s,x(s,\xi(z),0),0)\right\rangle \\
 \ \ \ \ \ \  \ \ \ \ \ \  \ \ \ \ \ \ ...\\
\left\langle u_k(s,z)\right., \left. g(s,x(s,\xi(z),0),0)\right\rangle
\end{array}\right)ds.
$$
Then the following statements hold.
\begin{itemize}
\item[1)] For any sequences $(\varphi_m)_{m\geq 1}$ from $C^0(\R,\R^n)$ and $(\e_m)_{m\geq 1}$ from $[0,1]$ such that $\varphi_m(0)\rightarrow \xi(z_0)\in \xi(\overline{U})$, $\e_m\rightarrow 0$ as $m\rightarrow \infty$ and $\varphi_m$ is a $T$-periodic solution of \eqref{perturbado} with $\e=\e_m$, we have that $M(z_0)=0$.
\item[2)] If $M(z)\neq 0$ for any $z\in \partial U$ and $d(M,U)\neq 0$, then there exists $\e_1>0$ sufficiently small such that for each $\e\in(0,\e_1]$ there is at least one $T$-periodic solution $\varphi_{\e}$ of system \eqref{perturbado} such that $\rho(\varphi_\e(0),\xi(\overline{U}))\rightarrow 0$ as $\e\rightarrow 0$, where $\rho(\varphi_\e(0),\xi(\overline{U}))= \min_{\zeta\in \xi(\overline{U})}\left\|\varphi_\e(0)-\zeta\right\|$ and $\left\|.\right\|$ is a norm in $\R^n$.
\end{itemize}
In addition we assume that there exists $z_0\in U$ such that $M(z_0)=0$, $M(z)\neq 0$ for all $z\in\overline{U}\backslash \{z_0\}$ and the Brouwer degree $d(M,U)$ of $M$ in $U$ satisfies $d(M,U)\neq 0$. Moreover, calling $w_{0}=\xi(z_0)$, we assume that:

\begin{itemize}
\item[iii)] For $\delta >0$ sufficiently small there exists $M_{\delta}\subset [0,T]$ Lebesgue measurable with $mes(M_{\delta})=\tilde{o}(\delta)$ such that
\begin{multline*}
\left\|g(t,w_1+\zeta,\e)-g(t,w_1,0)-g(t,w_2+\zeta,\e)+g(t,w_2,0)\right\|\\ 
\leq \tilde{o}(\delta)\left\|w_1-w_2\right\|,
\end{multline*}
for all $t\in [0,T]\backslash M_{\delta}$ and for all $w_1$, $w_2 \in B_{\delta}(w_{0})$, $\e\in [0,\delta]$ and $\zeta\in B_{\delta}(0)$.
\item[iv)] There exists $\delta_1>0$ and $L_M>0$ such that
\[\left\|M(z_1)-M(z_2)\right\|\geq L_M \left\|z_1-z_2\right\|, \ \mbox{for all} \ z_1, z_2 \in B_{\delta_1}(z_0).\]
\end{itemize}
Then the following conclusion holds.
\begin{itemize}
\item[3)] There exists $\delta_2>0$ such that for any $\e\in(0,\e_1]$, $\varphi_{\e}$ is the only $T$-periodic solution of \eqref{perturbado} with initial condition in $B_{\delta_2}(w_0)$. Moreover $\varphi_{\e}(0)\rightarrow \xi(z_0)$ as $\e\rightarrow 0$.
\end{itemize}
\end{theorem}

The proof of Theorem \ref{teotool} can be found in \cite{BLM} (Theorem 7, pag. 3916). For the case where $f\equiv0$ a similar result for $C^{0}$ functions using Brouwer degree can be found in \cite{BL}.

\begin{remark}\label{iii-v}
Since condition $(iii)$ is rather technical, instead of use it, in this paper we consider a simpler condition for the function $g$, as follows:
\begin{itemize}
\item[v)] For any $\lambda>0$ sufficiently small there exists $M_{\lambda}\subset[0,T]$ Le\-bes\-gue measurable with mes$(M_{\lambda})=\mathit{o}(\lambda)/\lambda$ and such that for every $t\in[0,T]\setminus M_{\lambda}$ and for all $w\in B_{\delta}(w_{0})$, $\varepsilon\in[0,\lambda]$, $||D_{w}g(t,w,\varepsilon)-D_{w}g(t,w_{0},0)||\leq\mathit{o}(\lambda)/\lambda$.
\end{itemize}
The condition $v)$ is a sufficient one for $iii)$. This fact follows from the Main Value Theorem.
\end{remark}

In this paper we apply Theorem \ref{teotool} in order to achieve the results above. As we commented before, there is, as well as we know, no other applications of such Theorem in the literature.

\smallskip

In order to get our results we choose to work with a regularization of system \eqref{system_noncontinuous} since its perturbed part is non-smooth instead of $C^{0}$. The regularization method was introduced in \cite{ST}. In the next lines we briefly summarize it. Indeed, consider $\mathcal{D}$ be an open subset of $\mathbb{R}^{n}$ and $F:\mathbb{R}\times\mathcal{D}\longrightarrow\mathbb{R}$ a $C^{1}$ function having $0$ as a regular value with $\Sigma=F^{-1}({0})$. A continuous function $\varphi:\mathbb{R}\longrightarrow\mathbb{R}$ is a transition function if $\varphi(t)=-1$ for $t\leq1$, $\varphi'(t)>0$ for $t\in(-1,1)$ and $\varphi(t)=1$ for $t\geq1$. So, for $\delta\in(0,1]$ we say that the one-parameter family of continuous functions $Z_{\delta}$ given by
$$
Z_{\delta}(t,X)=\dfrac{Y_{1}(t,X)+Y_{2}(t,X)}{2}+\varphi\left(\dfrac{f(t,X)}{\delta}\right)\dfrac{Y_{1}(t,X)-Y_{2}(t,X)}{2},
$$
is a $\varphi$-regularization of a non-smooth vector field $Z=(Y_{1},Y_{2})$, where $X\in\mathcal{D}$. In this paper, we obtain the results firstly for the regularized system $Z_{\delta}$ of system \eqref{system_noncontinuous} via Theorem \ref{teotool} and then we adapt such results by doing $\delta\to 0$, as we will see in the next section.

\section{Statement of the main results}\label{results}

\subsection{Preliminary of the results}\label{preliminares}

In this subsection we perform a regularization of system \eqref{system_noncontinuous} and also introduce two important functions in order to state the results. Indeed, first we identify $X=(x,y,z)$, $\mathcal{D}=\mathbb{R}^{3}$ and $F(t,X)=y$. If we consider the $C^{0}$ transition function
\begin{equation}\label{Fi}
\varphi_{\delta}(t)=\left\{\begin{array}{rl}
-1, & \mbox{if} \ t\leq-\delta, \vspace{0.2cm}\\
\dfrac{t}{\delta}, & \mbox{if} \ -\delta<t<\delta, \vspace{0.2cm}\\
1, & \mbox{if} \ t\geq\delta,
\end{array}\right.
\end{equation}
then a $C^{0}$ $\varphi_{\delta}$-regularization of system \eqref{system_noncontinuous} writes
\begin{equation}\label{system_regularised}
\begin{array}{rl}
\dot{X^{\varepsilon}_{\delta}}&=f^{\varepsilon}_{\delta}(t,X,\varepsilon)\vspace{0.3cm}\\
&=\dfrac{f^{+}(t,X,\varepsilon)+f^{-}(t,X,\varepsilon)}{2}+\varphi_{\delta}(y)\dfrac{f^{+}(t,X,\varepsilon)-f^{-}(t,X,\varepsilon)}{2}\vspace{0.3cm}\\
&=f(t,X)+\varepsilon g_{\delta}(t,X,\varepsilon),
\end{array}
\end{equation}
where
\begin{equation}\label{system_final}
g_{\delta}(t,X,\varepsilon)=\left\{\begin{array}{ll}
g^{-}(X), &  \ y\leq-\delta, \vspace{0.3cm}\\
\dfrac{g^{+}(X)+g^{-}(X)}{2}+\dfrac{y}{\delta}\left(\dfrac{g^{+}(X)-g^{-}(X)}{2}\right), &  |y|<\delta, \vspace{0.3cm}\\
g^{+}(X), & \ y\geq\delta.
\end{array}\right.
\end{equation}

\noindent We must stress that system \eqref{system_regularised} is smooth and it has the same unperturbed part of system \eqref{system_noncontinuous}, i.e., system \eqref{system_regularised} also possesses the cylinder $\textsl{C}$ filled by periodic solutions when $\varepsilon$ is zero and for all $\delta>0$. In addition, taking $\delta\rightarrow 0$ in system \eqref{system_regularised} we obtain the non-smooth system \eqref{system_noncontinuous}.
\begin{figure}[!h]
\begin{center}
\psfrag{a}{$\delta$}\psfrag{b}{$1$}\psfrag{c}{$-\delta$}\psfrag{d}{$-1$}\psfrag{e}{$1$}\psfrag{f}{$-1$}
 \epsfxsize=10cm
\epsfbox{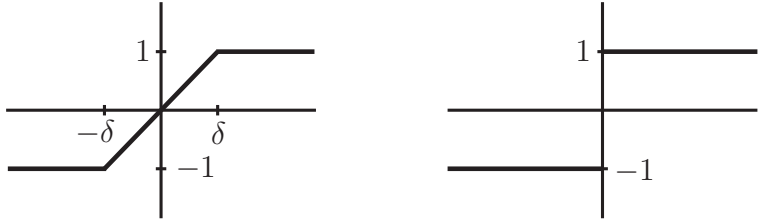} \caption{\small{The graph of $\varphi_{\delta}$ for $\delta>0$ (left) and for $\delta\to 0$ (right).}} \label{funcao_trasicao}\end{center}
\end{figure}
As we said before, in this paper we will apply Theorem \ref{teotool} and then we take $\delta\to0$ in order to extend the results for system \eqref{system_noncontinuous}.

\smallskip

Now we introduce two functions depending on the function $h$ (which determines the shape of the periodic solutions on the cylinder $\textsl{C}$) and the perturbations $p^{\pm}$, $q^{\pm}$ and $r^{\pm}$ which are very important to the results. Indeed, first consider the function
\begin{equation}
A_{h}(\theta)=\cos\theta\dfrac{\partial h}{\partial x}(\cos\theta,\sin\theta)+\sin\theta\dfrac{\partial h}{\partial y}(\cos\theta,\sin\theta).
\end{equation}

Observe that $A_{h}(\theta)$ is rather technical but it plays an important role in the implementation of Theorem \ref{teotool} for system \eqref{system_regularised}. Note also that it depends only on function $h$.

\smallskip

Next, let $M_{\delta}:\mathbb{R}\longrightarrow\mathbb{R}$ be a function defined by
\begin{equation}\label{formula_principal}
\begin{array}{rl}
M_{\delta}(z)=&\displaystyle\int_{0}^{2\pi}-\dfrac{1}{2}\left[h(\cos\theta,\sin\theta)(-\cos\theta(q^{+}(\varsigma)+q^{-}(\varsigma))\right.\vspace{0.2cm}\\
&\qquad +\sin\theta(p^{+}(\varsigma)+p^{-}(\varsigma))) +(r^{+}(\varsigma)+r^{-}(\varsigma)) + \vspace{0.2cm}\\
&\left. \qquad (h(\cos\theta,\sin\theta)(\cos\theta(-q^{+}(\varsigma)+q^{-}(\varsigma))+ \sin\theta(p^{+}(\varsigma)- \right.\vspace{0.2cm}\\
&\left. \qquad p^{-}(\varsigma))) +(r^{+}(\varsigma)-r^{-}(\varsigma)))\varphi_{\delta}(\sin\theta)\right]ds,\vspace{0.2cm}
\end{array}
\end{equation}
where $\varsigma=\left(\cos\theta,\sin\theta,z+\textstyle\int_{0}^{s}h(\cos v,\sin v)dv \right)$ and $z$ is some real value. We will see in Subsection \ref{main results} that under suitable assumptions, the simple zeros of function $M_{\delta}$ provide limit cycles bifurcating from the continuum of periodic solutions on the cylinder $\textsl{C}$.

\smallskip

Now we establish the main results.

\subsection{Main Results}\label{main results}

\begin{lemma}\label{fl}
Suppose that $A_{h}(\theta)=0$, $\forall\theta\in[0,2\pi)$. Then, for $|\e|$ sufficiently small and for every $z_{0}$ such that $M_{\delta}(z_{0})=0$ and $M_{\delta}'(z_{0})\neq0$, the smooth system \eqref{system_regularised} has a limit cycle bifurcating from the band of periodic solutions of the cylinder $\textsl{C}$ with $\e=0$. Moreover, there exists at most $s=\max\{m,n,p\}$ values of $z$ for which $M_{\delta}(z)=0$.
\end{lemma}

\begin{remark}
We note that the condition $A_{h}(\theta)=0$ is not empty. For instance, it is not difficult to verify that for each $c_{i}\in\mathbb{R}$ the functions
 $$
 \tilde{h}(x,y)=\displaystyle\sum_{i=0}^{\infty}\dfrac{c_{i}}{(x^{2}+y^{2})^{\frac{2i+1}{2}}}x^{2i+1}
 $$
satisfy such property. Moreover, it holds that the function $\rho(x,y)=1/(x^{2}+y^{2})^{\frac{2i+1}{2}}$ satisfies $\rho(r\cos\theta,r\sin\theta)=1$ and the power of $x$ in the polynomials $c_{i}x^{2i+1}$ is always odd, i.e., the cylinders defined by the function $\tilde{h}$ presented previously is filled by periodic orbits of system \eqref{system_no_perturbation}. Actually, this facts says that the results take into account infinitely many different cylinders.
\end{remark}

The next theorem is the main result of the paper and says that the limit cycles that we find for the regularized system \eqref{system_regularised} are preserved for the non-smooth system \eqref{system_noncontinuous} when $\delta\to 0$.

\begin{theorem}\label{t1}
Assume that $A_{h}(\theta)=0$, $\forall\theta\in[0,2\pi)$. Then, for $|\e|$ sufficiently small and for each $z_{0}$ such that $M_{\delta}(z_{0})=0$ and $M_{\delta}'(z_{0})\neq0$, the non-smooth system \eqref{system_noncontinuous} has a limit cycle bifurcating from the band of periodic solutions of the cylinder $\textsl{C}$ with $\e=0$. Moreover, there exists at most $s=\max\{m,n,p\}$ values of $z$ for which $M_{\delta}(z)=0$.
\end{theorem}


A particular case of perturbations of system \eqref{system_no_perturbation} is to consider $g^{+}=g^{-}$ in \eqref{perturbations}, i.e., perform the same perturbation in $\Sigma^{+}$ and $\Sigma^{-}$. In this particular case we obtain a smooth perturbation of system \eqref{system_no_perturbation}. In such case system \eqref{system_noncontinuous} becomes smooth and it coincides with its regularized system \eqref{system_regularised}. The next theorem states the results in such si\-tu\-a\-tion.

\begin{theorem}\label{t2}
Assume that $A_{h}(\theta)=0$, $\forall\theta\in[0,2\pi)$, $g^{+}=g^{-}$ and consider the function
\begin{equation}\label{formula2}
\overline{M}_{\delta}(z)=\displaystyle\int_{0}^{2\pi}-\left[h(\cos\theta,\sin\theta)(-\cos\theta q^{+}(\varsigma)+\sin\theta p^{+}(\varsigma)) +r^{+}(\varsigma)\right]ds,
\end{equation}
where $\varsigma=\left(\cos\theta,\sin\theta,z+\textstyle\int_{0}^{s}h(\cos v,\sin v)dv \right)$ and $z$ is some real value. Then, for $|\e|$ sufficiently small and for each $z_{0}$ such that $\overline{M}_{\delta}(z_{0})=0$ and $\overline{M}_{\delta}'(z_{0})\neq0$, the smooth system \eqref{system_noncontinuous} has a limit cycle bifurcating from the band of periodic solutions of the cylinder $\textsl{C}$ with $\e=0$. Moreover, there exists at most $s=\max\{m,n,p\}$ values of $z$ for which $\overline{M}_{\delta}(z)=0$.
\end{theorem}

We observe that although Theorems \ref{t1} and \ref{t2} provide the same upper bound for the number of limit cycles bifurcating from $\textsl{C}$ by performing non-smooth and smooth functions, respectively, for concrete examples we may reach different upper bounds in each case. In fact, in similar situations usually non-smooth systems present more limit cycles than smooth ones. In Subsection \ref{exemplos} we will discuss this topic in more details through a specific example. Before that, in what follows we present the proof of the results.

\section{Proofs and examples}\label{provas_discussoes}

\subsection{Proof of the results}\label{provas}

Now we apply the methods and tools described previously in order to prove the results presented in Subsection \ref{main results}. We start proving the Fundamental Lemma.

\begin{proof}[Proof of Fundamental Lemma:]
Consider system \eqref{system_regularised} and assume that for this system we verify $A_{h}(\theta)=0$ for all $\theta\in[0,2\pi)$. Since the periodic solutions of system \eqref{system_no_perturbation}, that we are perturbing, live on the cylinder $\textsl{C}$, we will perform a cylindrical change of coordinates in system \eqref{system_regularised} by introducing the new variables $(z,r,\theta)$ given implicitly by $x=r\cos\theta$, $y=r\sin\theta$ and $z=z$. In the new variables $(z,r,\theta)$ system \eqref{system_regularised} writes
\begin{equation}\label{system_cylindrical_coordinates}
\begin{array}{rl}
\dot{z}=&h(r\cos\theta,r\sin\theta)+\varepsilon\dfrac{1}{2}\left[ r^{+}(\vartheta)+r^{-}(\vartheta)+(r^{+}(\vartheta)-r^{-}(\vartheta))\right.\vspace{0.2cm}\\
&\left. \varphi_{\delta}(r\sin\theta)\right],\vspace{0.2cm}\\
\dot{r}=&-r+r^{3} +\varepsilon \dfrac{1}{2}\left[   \cos\theta(p^{+}(\vartheta)+p^{-}(\vartheta))+\sin\theta(q^{+}(\vartheta)+q^{-}(\vartheta))   +\right. \vspace{0.2cm}\\
&\left. \cos\theta(p^{+}(\vartheta)-p^{-}(\vartheta))+\sin\theta(q^{+}(\vartheta)-q^{-}(\vartheta))\varphi_{\delta}(r\sin\theta)    \right],\vspace{0.2cm}\\
\dot{\theta}=& 1+\varepsilon \dfrac{1}{2r}\left[   \cos\theta(q^{+}(\vartheta)+q^{-}(\vartheta))-\sin\theta(p^{+}(\vartheta)+p^{-}(\vartheta))   +\right. \vspace{0.2cm}\\
&\left. \cos\theta(q^{+}(\vartheta)-q^{-}(\vartheta))+\sin\theta(-p^{+}(\vartheta)+p^{-}(\vartheta))\varphi_{\delta}(r\sin\theta)    \right],
\end{array}
\end{equation}
where $\vartheta=(r\cos\theta,r\sin\theta,z)$.

\smallskip

Now we change the independent variable $t$ of system \eqref{system_cylindrical_coordinates} to the new variable $\theta$ and obtain the following equivalent system
\smallskip
\begin{equation}\label{system_cylindrical_coordinates_2}
\begin{array}{rcl}
\dfrac{dz}{d\theta}&=& h(r\cos\theta,r\sin\theta)+\varepsilon\dfrac{1}{2r}\left[h(r\cos\theta,r\sin\theta)(-\cos\theta(q^{+}(\vartheta)+\right.\vspace{0.2cm}\\
&& q^{-}(\vartheta))+\sin\theta(p^{+}(\vartheta)+p^{-}(\vartheta))) +r(r^{+}(\vartheta)+r^{-}(\vartheta)) + \vspace{0.2cm}\\
&&\left. (h(r\cos\theta,r\sin\theta)(\cos\theta(-q^{+}(\vartheta)+q^{-}(\vartheta))+ \sin\theta(p^{+}(\vartheta)- \right.\vspace{0.2cm}\\
&&\left. p^{-}(\vartheta))) +r(r^{+}(\vartheta)-r^{-}(\vartheta)))\varphi_{\delta}(r\sin\theta)\right]+\textsl{O}_2\vspace{0.2cm}\\
&=& h(r\cos\theta,r\sin\theta)+\varepsilon G_{\delta}^{1}(\theta,z,r,\varepsilon),\vspace{0.2cm}\\
\dfrac{dr}{d\theta}&=&-r+r^{3} +\varepsilon \dfrac{1}{2}\left[   -(q^{+}(\vartheta)+q^{-}(\vartheta))((r^{2}-1)\cos\theta-\sin\theta)+  \right. \vspace{0.2cm}\\
&&  (p^{+}(\vartheta)+p^{-}(\vartheta))((r^{2}-1)\sin\theta+\cos\theta)+(-(q^{+}(\vartheta)+q^{-}(\vartheta))  \vspace{0.2cm}\\
&& ((r^{2}-1)\cos\theta-\sin\theta)+(p^{+}(\vartheta)+p^{-}(\vartheta))((r^{2}-1)\sin\theta+ \vspace{0.2cm}\\
&&\left.\cos\theta))\varphi_{\delta}(r\sin\theta)\right]+\textsl{O}_2 \vspace{0.2cm}\\
&=& -r+r^{3} +\varepsilon G_{\delta}^{2}(\theta,z,r,\varepsilon),
\end{array}
\end{equation}
where again $\vartheta=(r\cos\theta,r\sin\theta,z)$ and $\textsl{O}_2=\textsl{O}(\varepsilon^{2})$. Observe that the vector field of system \eqref{system_cylindrical_coordinates_2} is $2\pi$-periodic. Addi\-ti\-o\-na\-lly, in order to see that its perturbed part is locally uniformly Lipschitz in the variables $(z,r)\in\mathbb{R}^{2}$, consider the function $G_{\delta}:\mathbb{\mathbb{R}}\times\mathbb{R}^{2}\times[0,1]\longrightarrow\mathbb{R}^{2}$ as $G_{\delta}(\theta,w_{1},w_{2},\varepsilon)=(G_{\delta}^{1}(\theta,w_{1},w_{2},\varepsilon),G_{\delta}^{2}(\theta,w_{1},w_{2},\varepsilon))$. Consider also the sets
$$
\begin{array}{l}
R_{1}=\{(\theta,w_{1},w_{2},\varepsilon)\in\mathbb{\mathbb{R}}\times\mathbb{R}^{2}\times[0,1];w_{2}\sin\theta\leq -\delta\},\vspace{0.2cm}\\
R_{2}=\{(\theta,w_{1},w_{2},\varepsilon)\in\mathbb{\mathbb{R}}\times\mathbb{R}^{2}\times[0,1];-\delta\leq w_{2}\sin\theta\leq \delta\},\vspace{0.2cm}\\
R_{3}=\{(\theta,w_{1},w_{2},\varepsilon)\in\mathbb{\mathbb{R}}\times\mathbb{R}^{2}\times[0,1];w_{2}\sin\theta\geq \delta\},
\end{array}
$$
and let $K\subset\mathbb{\mathbb{R}}\times\mathbb{R}^{2}\times[0,1]=\bigcup_{i=1,2,3}R_{i}$ be a compact set. In order to see that $G_{\delta}$ is Lipschitz on $K$, it is sufficient to show that $G_{\delta}$ is Lipschitz on the convex hull $\overline{K}_{C}$ of $K$, once $K\subseteq \overline{K}_{C}$. Indeed, let $x$ and $y$ be two arbitrary points of $\mathbb{R}^{2}$ in $\overline{K}_{C}$. Now consider $S$ the segment connecting $x$ and $y$ and $S_{i}=S\cap R_{i}$, $i=1,2,3$. This intersection consists of a finite number of closed segments contained in $S$, since the boundaries between $R_{1}$ and $R_{2}$ and between $R_{2}$ and $R_{3}$ are codimension one manifolds of $\mathbb{\mathbb{R}}\times\mathbb{R}^{2}\times[0,1]$. The restrictions $G_{\delta}|_{S_{i}}$ are polynomial in the variables $w_{1}$ and $w_{2}$ for each $i=1,2,3$ and consequently they are also $C^{\infty}$. It means that each restriction $G_{\delta}|_{S_{i}}$ is locally $L_{i}$-Lipschitz on the compact set $S_{i}$, for each $i=1,2,3$, which is equivalent to be $L_{i}$-Lipschitz. Then for all $(\theta,x,\varepsilon),(\theta,y,\varepsilon)\in \overline{K}_{C}$, there exists $L=\max\{L_{1},L_{2},L_{3}\}$ such that
\begin{multline*}
||G_{\delta}(\theta,x,\varepsilon)-G_{\delta}(\theta,y,\varepsilon)|| \leq \displaystyle\sum_{p^{i}\in S_{i}}L_{i}||p_{i}-p_{i+1}||  \\
\leq L\left( \displaystyle\sum_{p^{j}\in S_{1}}||p_{j}-p_{j+1}||+\displaystyle\sum_{p^{k}\in S_{2}}||p_{k}-p_{k+1}||+\displaystyle\sum_{p^{l}\in S_{3}}||p_{l}-p_{l+1}|| \right),
\end{multline*}
where $p^{s}$ is the segment with ends in $p_{s}$ and $p_{s+1}$ for $s\in\{i,j,k,l\}$, $x=p_{s}$ and $y=p_{r}$, for some $s,r\in\mathbb{N}$. Consequently, if $n$ is the number of intersections of $S$ with the boundaries of each $R_{i}$, $i=1,2,3$, then once $S$ is a segment we obtain
$$
\begin{array}{rcl}
||G_{\delta}(\theta,x,\varepsilon)-G_{\delta}(\theta,y,\varepsilon)||&\leq& L\left(||x-p_{1}||+\ldots+||p_n-y||\right)\vspace{0.2cm}\\
&\leq & L||x-y||.
\end{array}
$$

Hence $G_{\delta}$ is locally uniformly Lipschitz in the variables $(z,r)\in\mathbb{R}^{2}$.

\smallskip

Now we call $X=(z,r)$ and consider system \eqref{system_cylindrical_coordinates_2} with $\varepsilon=0$. Then we obtain
\begin{equation}\label{unperturbed2}
\begin{array}{l}
\dfrac{dX}{d\theta}=f(\theta,X),
\end{array}
\end{equation}
where $f(\theta,X)=(h(r\cos\theta,r\sin\theta),-r+r^3)$. By hypothesis $f\in C^{2}$. In the $zr$-plane, the straight line $r=1$ is invariant. Hence the solution $X(\theta,X_{0},0)$ with initial condition $X_{0}=(z_{0},1)$ is
$$
X(\theta,X_{0},0)=(z(\theta,z_{0}),r(\theta,z_{0}))=\left(z_{0}+\displaystyle\int_{0}^{\theta}h(\cos s,\sin s)ds,1\right).
$$

\smallskip

Since $\textstyle\int_{0}^{\theta}h(\cos s,\sin s)ds$ is periodic, it follows that $z(\theta,z_{0})$ is $2\pi$-periodic in the variable $\theta$ and for each point in a neighborhood of $z=z_{0}$ on the straight line $r=1$ passes a $2\pi$-periodic solution that lies in the phase space $(z,r,\theta)\in\mathbb{R}^2\times\mathbb{S}^{1}$. Consequently system \eqref{unperturbed2} has a family $\mathcal{M}=\{(z,r)\in\mathbb{R}^{2}:r=1\}$ of $2\pi$-periodic solutions. 

\smallskip

Now consider $R_{0}>0$. There exists an open ball $U\subset\mathbb{R}$, $R_{0}\in\mathbb{R}$, $U=\{z_{0}\in(-R_{0},R_{0})\}$ and a function $\xi\in C^{1}(\overline{U},\mathbb{R}^{2})$,
$$
\xi(z)=\left(z+\displaystyle\int_{0}^{\theta}h(\cos s,\sin s)ds,1\right),
$$
which is a parametrization of each periodic solution on $\mathcal{M}$ satisfying that for any $z\in\overline{U}$, we have $D\xi(z)=(1,0)^{T}$, whose rank is 1. Note that $\xi(z)$ is the initial condition of a $2\pi$-periodic solution of \eqref{unperturbed2}.

\smallskip

Now we linearize system \eqref{unperturbed2} along its periodic solutions $X(\theta,\xi(z),$ $0)$. We get
\begin{equation}\label{system_linearized}
\dfrac{dY}{d\theta}=D_{X}f(\theta,X(\theta,\xi(z),0))Y.
\end{equation}

Consequently the matrix $D_{X}f(\theta,X(\theta,\xi(z),0))$ writes
$$
\left(
\begin{array}{cc}
\dfrac{\partial u}{\partial z}(\xi(z),1)&\dfrac{\partial u}{\partial r}(\xi(z),1)\vspace{0.2cm}\\
\dfrac{\partial (-r+r^{3})}{\partial z}(\xi(z),1)&\dfrac{\partial (-r+r^{3})}{\partial r}(\xi(z),1)
\end{array}
\right),
$$
with $u(r,\theta)=h(x(r,\theta),y(r,\theta))$, $x(r,\theta)=r\cos\theta$ and $y(r,\theta)=r\sin\theta$

\smallskip

Now, if we observe that $\frac{\partial u}{\partial r}(\xi(z),1)=A_{h}(\theta)$, then it is easy to check that system \eqref{system_linearized} writes
\begin{equation}
\dfrac{dY}{d\theta}=\left(
\begin{array}{cc}
0&A_{h}(\theta)\\
0&2
\end{array}
\right)Y,
\end{equation}
and has the fundamental matrix $N_ {Y}(\theta)$ given by
\begin{equation}
N_{Y}(\theta)=\left(
\begin{array}{cc}
1&\displaystyle\int_{0}^{\theta}e^{2s}A_{h}(s)ds\\
0&e^{2\theta}
\end{array}
\right).
\end{equation}

We note that $N_{Y}(0)=I_{2}$. Thus, since $A_{h}(\theta)=0$ by hypothesis, the monodromy matrix $C=N_{Y}^{-1}(0)N_{Y}(2\pi)$ is
\begin{equation}
C=\left(
\begin{array}{cc}
1&0\\
0&e^{4\pi}
\end{array}
\right),
\end{equation}
and consequently system  \eqref{system_linearized} has the Floquet multiplier $+1$ with the geometric multiplicity equal to 1.

\smallskip

In what follows we consider the adjoint linear system
\begin{equation}\label{system_adjoint}
\dfrac{dU}{d\theta}=-(D_{X}f(\theta,X(\theta,\xi(z),0)))^{*}U.
\end{equation}

Since system \eqref{system_adjoint} is the adjoint of system \eqref{system_linearized}, its fundamental matrix $N_{U}(\theta)$ is $N_{U}(\theta)=-(N_{Y}(\theta))^{*}$ and consequently a linearly independent solution is $u_{1}(\theta,z)=(-1,0)^{T}$. Observe that $u_{1}(0,z)$ is $C^{1}$ with respect to $z$. Therefore the Malkin's bifurcation function $M_{\delta}(z)$ takes the form
\begin{equation}
\begin{array}{rcl}
M_{\delta}(z)&=&\displaystyle\int_{0}^{2\pi}<u_{1}(s,z),G_{\delta}(s,X(s,\xi(z),0))>ds\vspace{0.2cm}\\
&=&\displaystyle\int_{0}^{2\pi}-G_{\delta}^{1}(s,X(s,\xi(z),0))ds.
\end{array}
\end{equation}

In other words, we obtain the formula
\begin{equation}\label{formula0}
\begin{array}{rl}
M_{\delta}(z)=&\displaystyle\int_{0}^{2\pi}-\dfrac{1}{2}\left[h(\cos\theta,\sin\theta)(-\cos\theta(q^{+}(\varsigma)+q^{-}(\varsigma))\right.\vspace{0.2cm}\\
&\qquad +\sin\theta(p^{+}(\varsigma)+p^{-}(\varsigma))) +(r^{+}(\varsigma)+r^{-}(\varsigma)) + \vspace{0.2cm}\\
&\left. \qquad (h(\cos\theta,\sin\theta)(\cos\theta(-q^{+}(\varsigma)+q^{-}(\varsigma))+ \sin\theta(p^{+}(\varsigma)- \right.\vspace{0.2cm}\\
&\left. \qquad p^{-}(\varsigma))) +(r^{+}(\varsigma)-r^{-}(\varsigma)))\varphi_{\delta}(\sin\theta)\right]ds,\vspace{0.2cm}
\end{array}
\end{equation}
where now $\varsigma=\left(\cos\theta,\sin\theta,z+\textstyle\int_{0}^{s}h(\cos v,\sin v)dv \right)$.

\smallskip

In order to use Theorem \ref{teotool} to assure the existence of limit cycles for system \eqref{system_regularised}, we observe that for each $z_{0}\in U$ such that $M_{\delta}(z_{0})=0$ and $M_{\delta}'(z_{0})\neq0$, the Implicit Function Theorem says that $M_{\delta}'(z)\neq0$ for all $z\in\overline{U}$, and then we get $d(M_{\delta},U)\neq0$ since $M_{\delta}$ is continuous. In addition, we must verify condition $iii)$ of Section 2. However, taking into account Remark \ref{iii-v}, we will verify condition $v)$ instead of condition $iii)$. Indeed, let $\lambda$ be a positive number, $\varepsilon\in[0,\lambda]$, $w_{0}=\xi(z_{0})$ and consider the values $p_{+}=\arcsin(\delta/r)$ and $p_{-}=\arcsin(-\delta/r)$. Observe that function $\varphi_{\delta}$ is continuous except in the points $\theta=p_{\pm}$. In addition, consider the sets $L_{\lambda}^{\pm}=\{\omega\in[0,2\pi];|\omega-p_{\pm}|<2\delta\}$ and $L_{\lambda}=L_{\lambda}^{-}\cup L_{\lambda}^{+}$. Thus $med(L_{\lambda})=8\lambda=\mathit{o}(\lambda)/\lambda$.

\smallskip

Now we observe that for all $\theta\in[0,2\pi]\setminus L_{\lambda}$ and for all $w\in B_{\lambda}(w_{0})$ the function $G_{\delta}$ is $C^{\infty}$. Indeed, $G_{\delta}$ does not switch from one region $R_{i}$ to another $R_{j}$ when $i\neq j$ and $w$ varies on $B_{\lambda}(w_{0})$, for $i,j=1,2,3$. It holds once $\theta\in[0,2\pi]\setminus L_{\lambda}$ and the radius of the ball $B_{\lambda}(w_{0})$ is smaller than the radius of each neighborhood $(p_{\pm}-\lambda,p_{\pm}+\lambda)$ of $p_{\pm}$. Then, from the Mean Value Theorem we obtain
$$
||D_{w}G_{\delta}(\theta,w,\varepsilon)-D_{w}G_{\delta}(\theta,w_{0},0)||\leq \displaystyle\sup_{s\in\overline{S}}||D^{2}_{w}G_{\delta}(s)||\cdot|w-w_{0}|,
$$
where $S=B_{\lambda}(w_{0})$ and $D^{2}_{w}G_{\delta}$ denotes the second derivative of the function $D_{w}G_{\delta}$. Thus, once $D^{2}_{w}G_{\delta}$ is $C^{1}$ on the compact set $V$, it holds
$$
||D_{w}G_{\delta}(\theta,w,\varepsilon)-D_{w}G_{\delta}(\theta,w_{0},0)||\leq K|w-w_{0}|\leq K\lambda=\dfrac{K\lambda^{2}}{\lambda},
$$
and then condition $v)$ holds once $K\lambda^2=\mathit{o}(\lambda)$.

\smallskip

In order to prove condition $iv)$, first observe that $M_{\delta}$ is a real-valued function whose domain is $\mathbb{R}$. Now, given $z_{0}\in U$ satisfying $M_{\delta}(z_{0})=0$ with $M_{\delta}'(z)\neq0$ for all $z\in\overline{U}$, consider $\delta_{1}>0$ an arbitrary value such that $V=(z_{0}-\delta_{1},z_{0}+\delta_{1})\subset U$. Thus, by the Mean Value Theorem, there exist $c\in V$ such that
$$
|M_{\delta}(b)-M_{\delta}(a)|=|M_{\delta}'(c)|\cdot|b-a|,
$$
for all $a,b\in V$ and $c\in(a,b)$. The proof of condition $iv)$ follows taking $L_{M_{\delta}}=\inf\{|M_{\delta}'(z)|;z\in V\}>0$ and observing that $|M_{\delta}'(c)|\geq L_{M_{\delta}}$, i.e., $|M_{\delta}(b)-M_{\delta}(a)|\geq L_{M_{\delta}}|b-a|$ for all $a,b\in V$.

\smallskip

Therefore Theorem \ref{teotool} assures that there exists $\varepsilon_{1}>0$ sufficiently small and $\delta_{2}>0$ such that for each $\varepsilon\in(0,\varepsilon_{1})$, there exist a unique $2\pi$-periodic solution (consequently a limit cycle) $\varphi_{\delta}^{\varepsilon}\in\C^{0}(\mathbb{R},\mathbb{R}^{2})$ of the regularized system \eqref{system_cylindrical_coordinates_2} with condition in $B_{\delta_{2}}(\xi(z_{0}))$ satisfying $\varphi_{\delta}^{\varepsilon}(0)\to\xi(z_{0})$ when $\varepsilon\to 0$. Consequently the equivalent systems \eqref{system_cylindrical_coordinates} and \eqref{system_regularised} also posses the limit cycle $\varphi_{\delta}^{\varepsilon}(t)$ satisfying such properties.

\smallskip

Observe that in the particular case treated in this paper, the regularized system \eqref{system_regularised} with $\varepsilon=0$ does not depend on $\delta$. Thus, neither the initial condition $\xi(z)$ and consequently nor $\varphi_{\delta}^{\varepsilon}$ depends on $\delta$. Moreover, once $\xi(z_{0})=(z_{0}+\textstyle\int_{0}^{\theta}h(\cos s,\sin s)ds,1)$ has the second component equal to one (what means $r=1$, in the cylindrical coordinates), the limit cycle $\varphi_{\delta}^{\varepsilon}(t)$ lives on the cylinder $\textsl{C}$, i.e., $\varphi_{\delta}^{\varepsilon}(t)$ bifurcates from the continuum of periodic solutions on $\textsl{C}$.

\smallskip

Finally, replacing the expressions of $p^{\pm}$, $q^{\pm}$ and $r^{\pm}$ given in \eqref{perturbations} into the expression \eqref{formula_principal}, we obtain the polynomial
\begin{equation}\label{formula}
M_{\delta}(z)=I_{s}(\delta)z^{s}+\ldots+I_{1}(\delta)z+I_{0}(\delta),
\end{equation}
where $s=\max\{m,n,p\}$ and
$$
I_{j}(\delta)=\displaystyle\int_{0}^{2\pi}\phi_{j}(\theta,\delta)d\theta\in\mathbb{R},
$$
for some $\phi_{j}$ depending on $\theta$ and $\delta$ with $j=0,1,\ldots,s$.

Therefore, since $M_{\delta}(z)$ is a polynomial in $z$ possessing at most $s$ zeros, then $s=\max\{m,n,p\}$ is a upper bound for the number of zeros of $M_{\delta}$. But consequently, by using Theorem \ref{teotool}, $s$ is also the upper bound for the number of limit cycles that can bifurcate from the cylinder of system \eqref{system_cylindrical_coordinates_2}. Then it follows that the same holds for the equivalent system \eqref{system_regularised}. This finish the prove of the Fundamental Lemma.
\end{proof}

\smallskip

In what follows we prove Theorem \ref{t1}.
\begin{proof}[Proof of Theorem \ref{t1}]
Suppose that $A_{h}(\theta)=0$ for all $\theta\in[0,2\pi)$ and that for $|\e|$ sufficiently small we have a value $z_{0}$ such that $M_{\delta}(z_{0})=0$ and $M_{\delta}'(z_{0})\neq0$, where $M_{\delta}$ is given in \eqref{formula_principal}. Then, by the Fundamental Lemma, there exists a limit cycle $\varphi_{\delta}^{\varepsilon}(t)$ for system \eqref{system_regularised} satisfying $\varphi_{\delta}^{\varepsilon}(0)\to\xi(z_{0})$ when $\varepsilon\to 0$, as described in the proof of the Fundamental Lemma. Now consider $\Sigma^{z_{0}}$ a transversal section of $\varphi_{\delta}^{\varepsilon}$ contained in the cylinder $\textsl{C}$ for the Poincaré map $P^{\varepsilon}_{\delta}:\Sigma^{z_{0}}\longrightarrow\Sigma^{z_{0}}$, where $P^{\varepsilon}_{\delta}(z)=X_{\delta}(2\pi,\xi(z),\varepsilon)$, $\varphi_{\delta}^{\varepsilon}(0)\in\Sigma^{z_{0}}$ and $X_{\delta}(t,X,\varepsilon)$ is a solution of the regularized system \eqref{system_regularised}. Then it follows that $P^{\varepsilon}_{\delta}(\varphi_{\delta}^{\varepsilon}(0))=\varphi_{\delta}^{\varepsilon}(0)$.

\smallskip

Consider also $P^{\varepsilon}:\Sigma^{z_{0}}\longrightarrow\Sigma^{z_{0}}$ the Poincaré map of the non-smooth system \eqref{system_noncontinuous} with $P^{\varepsilon}_{\delta}(z)=X(2\pi,\xi(z),\varepsilon)$. Observe that by taking $\varepsilon$ sufficiently small the Poincaré map $P^{\varepsilon}$ is a composition of Poincaré maps of the regularized system and it is well defined and continuous for every $z\in\Sigma^{z_{0}}$. Moreover, each fix point of $P^{\varepsilon}$ corresponds to a periodic solution of the non-smooth system \eqref{system_noncontinuous}. Then it holds that $\displaystyle\lim_{\delta\to 0}P^{\varepsilon}_{\delta}(z)=P^{\varepsilon}(z)$, i.e., $P^{\varepsilon}$ is the pointwise limit of $P^{\varepsilon}_{\delta}$.

\smallskip

Therefore the point $\varphi^{\varepsilon}(0)=\textstyle\lim_{\delta\to 0}\varphi_{\delta}^{\varepsilon}(0)$ is a fixed point of the Poincaré map $P^{\varepsilon}(z)$ and consequently the non-smooth system \eqref{system_noncontinuous} has a limit cycle $\varphi^{\varepsilon}(t)$ such that $\varphi^{\varepsilon}(0)\to(\xi(z_{0}),1)$ when $\varepsilon\to0$.
\end{proof}

Finally we prove Theorem \ref{t2}. It is an immediate consequence of the Fundamental Lemma.

\begin{proof}[Proof of Theorem \ref{t2}:]
Since $g^{+}=g^{-}$, we obtain $p^{+}=p^{-}$, $q^{+}=q^{-}$ and $r^{+}=r^{-}$. Thus, $A_{h}(\theta)=0$ for all $\theta\in[0,2\pi)$, from formula \eqref{formula_principal} we get
\begin{equation}\label{formulaa}
\overline{M}_{\delta}(z)=\displaystyle\int_{0}^{2\pi}-\left[h(\cos\theta,\sin\theta)(-\cos\theta q^{+}(\varsigma)+\sin\theta p^{+}(\varsigma)) +r^{+}(\varsigma)\right]ds,
\end{equation}
where $\varsigma=\left(\cos\theta,\sin\theta,z+\textstyle\int_{0}^{s}h(\cos v,\sin v)dv \right)$. In addition, replacing the expressions of $p^{+}$, $q^{+}$ and $r^{+}$ given in \eqref{perturbations}, we obtain a polynomial
\begin{equation}\label{formula}
\overline{M}_{\delta}(z)=\overline{I}_{s}(\delta)z_{0}^{s}+\ldots+\overline{I}_{1}(\delta)z_{0}+\overline{I}_{0}(\delta),
\end{equation}
where again $s=\max\{m,n,p\}$ and
$$ \overline{I}_{j}(\delta)=\displaystyle\int_{0}^{2\pi}\overline{\phi}_{j}(\theta,\delta)d\theta\in\mathbb{R},
$$
for some $\overline{\phi}_{j}$ depending on $\theta$ and $\delta$ with $j=0,1,\ldots,s$.

The proof of Theorem \ref{t2} is straightforward from the Fundamental Lemma.
\end{proof}

We should mention that the formula obtained in \eqref{formulaa} does not coincides precisely to the one presented in \cite{LT} due to a subtle technical mistake performed in that paper. However, it is important to note that such misunderstanding does not affects the content of that paper since the goal of the authors was to present the methodology for computing limit cycles that bifurcate from a continuum of periodic orbits forming a subset of $\mathbb{R}^{n}$.

\smallskip

Next we present a concrete example and some particular perturbations of it in order to discuss some points about the results. More specifically, we compare the results of Theorems \ref{t1} and \ref{t2}.

\subsection{Examples}\label{exemplos}

In this subsection we present some considerations \linebreak about the number of periodic solutions that can bifurcate from a special cylinder (more specifically, we fix a function $h$) taking into account smooth and non-smooth perturbations. We must note that obtaining a global result about the achievement of the number of periodic solutions from formula \eqref{formula_principal} in terms of $h(x,y)$ and the values $m$, $n$ and $p$ is a hard task. Besides, we show that usually it is not possible neither reach the bound presented in Theorems \ref{t1} and \ref{t2} nor make the respective bounds coincide.

\smallskip

First consider $h(x,y)=x/(\sqrt{x^2+y^2})$. Thus system \eqref{system_no_perturbation} is defined in $\mathbb{R}^{3}\setminus\{(0,0,z);z\in\mathbb{R}\}$. We stress out that this particular case was studied in \cite{LT} by considering smooth perturbations. Observe that this particular function $h$ is $C^{2}$ and its expression in cylindrical coordinates is $h(\theta)=\cos\theta$. Moreover, it satisfies $\textstyle\int_{0}^{2\pi}h(\theta)d\theta=0$ and it is not difficult to see that $A_{h}(\theta)=0$ for all $\theta\in[0,2\pi)$. Now consider the perturbations in \eqref{perturbations} with $m=2$, $n=p=0$ and $a^{\pm}_{ijk}=0$ for all $i,j,k\in\mathbb{N}$ satisfying $i+j+k\leq 1$. Namely,
\begin{equation}\label{perturbations2}
\begin{array}{rcl}
p^{\pm}(x,y,z)&=& a^{\pm}_{200}x^{2}+a^{\pm}_{020}y^{2}+a^{\pm}_{002}z^{2}+a^{\pm}_{110}xy+a^{\pm}_{101}xz+a^{\pm}_{011}yz,\vspace{0.2cm}\\
q^{\pm}(x,y,z)&=& b^{\pm}_{000},\vspace{0.2cm}\\
r^{\pm}(x,y,z)&=& c^{\pm}_{000}.
\end{array}
\end{equation}
By using formula \eqref{formula2} we obtain
$$
\begin{array}{rcl}
\overline{M}_{\delta}(z)&=&\dfrac{\pi}{4}(a^{+}_{101}+a^{+}_{110}+4b^{+}_{000}+8c^{+}_{000}).
\end{array}
$$
Then $\overline{M}_{\delta}$ has no zero if $a^{+}_{101}+a^{+}_{110}+4b^{+}_{000}+8c^{+}_{000}\neq 0$ and a continuum of zeros otherwise, and consequently Theorem \ref{t2} does not provide any periodic solution bifurcating from the cylinder $\textsl{C}$ for these particular cases of perturbations and function $h$. On the other hand, now we use formula \eqref{formula2}, which provides the periodic solutions bifurcating from $\textsl{C}$ via non-smooth perturbations. Nevertheless, note that formula \eqref{formula2} depends on the function $\varphi_{\delta}(\sin\theta)$, then we need to apply a careful approach. Indeed, we will study this case in two steps. First, assume that $\delta\geq1$ and observe that in this situation we obtain $|\sin\theta|/\delta\leq 1$. Then $\varphi_{\delta}(\sin\theta)=\sin\theta/\delta$ and from formula \eqref{formula2} we get
$$
\begin{array}{rl}
M_{\delta}(z)=&\dfrac{\pi}{8}(a^{+}_{101}+a^{+}_{110}+4b^{+}_{000}+8c^{+}_{000}+a^{-}_{101}+a^{-}_{110}+4b^{-}_{000}+8c^{-}_{000})\vspace{0.2cm}\\
&+\dfrac{\pi}{8\delta}(a^{+}_{101}-a^{-}_{101})z.
\end{array}
$$

Observe that considering $a^{+}_{101}-a^{-}_{101}\neq 0$, function $M_{\delta}$ has exactly one zero $z_{0}$, namely,
$$
z_{0}=-\dfrac{(a^{+}_{101}+a^{+}_{110}+4b^{+}_{000}+8c^{+}_{000}+a^{-}_{101}+a^{-}_{110}+4b^{-}_{000}+8c^{-}_{000})\delta}{a^{+}_{101}-a^{-}_{101}},
$$
and $z_{0}$ satisfies $M_{\delta}'(z_{0})=\frac{\pi}{8\delta}(a^{+}_{101}-a^{-}_{101})\neq 0$.

\smallskip

Now suppose that $\delta<1$ and consider $\theta_{\delta}\in(0,\pi/2)$ such that $\sin\theta_{\delta}=\delta$, i.e., $\theta_{\delta}=\arcsin\delta$. Note that in order to use formula \eqref{formula2}, we must split limit of integration of the integral in pieces, taking into account the expression of $\varphi_{\delta}(\sin\theta)$ as follows.
$$
\begin{array}{lcc}
\varphi_{\delta}(\sin\theta)=1,& \mbox{for}& \delta\leq \sin\theta\leq 1,\vspace{0.2cm}\\
\varphi_{\delta}(\sin\theta)=\sin\theta/\delta,&  \mbox{for}& -\delta< \sin\theta< \delta,\vspace{0.2cm}\\
\varphi_{\delta}(\sin\theta)=-1,&  \mbox{for}& -1\leq \sin\theta\leq -\delta.
\end{array}
$$

Hence, the expression of $M_{\delta}$ is obtained by performing integral \eqref{formula2} with $\theta$ ranging in the partition $\{0,\theta_{\delta},\pi-\theta_{\delta},\pi+\theta_{\delta},2\pi-\theta_{\delta},2\pi\}$ of the interval $[0,2\pi]$ (see Figure \ref{exemlos_grafico_phi}).

\begin{figure}[!h]
\begin{center}
\psfrag{a}{$\theta_{\delta}$}\psfrag{b}{$\pi-\theta_{\delta}$}\psfrag{c}{$\pi+\theta_{\delta}$}\psfrag{d}{$2\pi-\theta_{\delta}$}\psfrag{e}{$2\pi$}\psfrag{f}{$0$}\psfrag{g}{$1$}\psfrag{h}{$\delta$}\psfrag{i}{$-\delta$}\psfrag{j}{$-1$}\psfrag{k}{$\sin\theta$}\psfrag{l}{$\theta$}
 \epsfxsize=8.3cm
\epsfbox{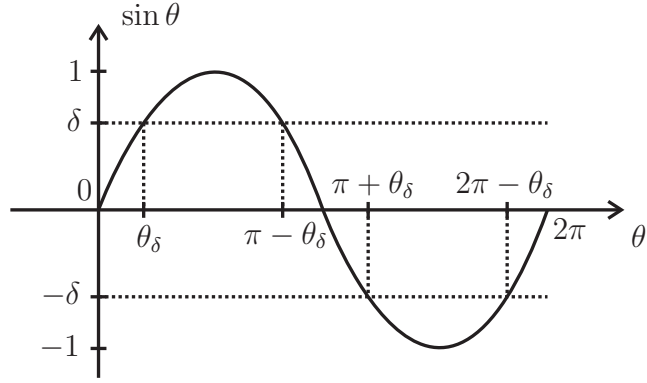} \caption{\small{Different intervals of integration of formula \eqref{formula}}.} \label{exemlos_grafico_phi}\end{center}
\end{figure}
Therefore, for $\delta<1$ we obtain the formula
$$
\begin{array}{rl}
M_{\delta}(z)=&\dfrac{\pi}{8}(a^{+}_{101}+a^{+}_{110}+4b^{+}_{000}+8c^{+}_{000}+a^{-}_{101}+a^{-}_{110}+4b^{-}_{000}+8c^{-}_{000})\vspace{0.2cm}\\
&-\dfrac{\delta\sqrt{1-\delta^{2}}(-5+2\delta^{2})-3\arcsin(\delta)}{12\delta}(a^{+}_{101}-a^{-}_{101})z,
\end{array}
$$
where we assume that $a^{+}_{101}-a^{-}_{101}\neq0$. In addition, one should note that the function $\Lambda(\delta)=\delta\sqrt{1-\delta^{2}}(-5+2\delta^{2})-3\arcsin(\delta)$ satisfies $\Lambda(0)=0$ and $\Lambda'(\delta)=-8(1-\delta^{2})^{3/2}<0$ for all $0<\delta<1$, i.e., $\Lambda(\delta)\neq0$ when $0<\delta<1$. Therefore function $M_{\delta}$ possesses the zero
$$
z_{0}=\dfrac{3\pi(a^{+}_{101}+a^{+}_{110}+4b^{+}_{000}+8c^{+}_{000}+a^{-}_{101}+a^{-}_{110}+4b^{-}_{000}+8c^{-}_{000})\delta}{2(\delta\sqrt{1-\delta^{2}}(-5+2\delta^{2})-3\arcsin(\delta))(a^{+}_{101}-a^{-}_{101})}.
$$

Moreover, it is easy to see that
$$
M_{\delta}'(z_{0})=-\dfrac{\delta\sqrt{1-\delta^{2}}(-5+2\delta^{2})-3\arcsin(\delta)}{12\delta}(a^{+}_{101}-a^{-}_{101})\neq 0,
$$
and consequently the Fundamental Lemma assures the existence of a limit cycle for system \eqref{system_regularised} considering the function $h(x,y)=x/$ $(\sqrt{x^2+y^2})$, perturbation \eqref{perturbations2} and $\varepsilon$ sufficiently small. However, by Theorem \ref{t1}, this periodic solution remains when $\delta$ tends to zero, in such sense that system \eqref{system_noncontinuous} has also a periodic solution. Indeed, when $\delta\to 0$ we achieve $\delta<1$ and then we get
$$
\begin{array}{rl}
M_{\delta}(z)=&\dfrac{\pi}{8}(a^{+}_{101}+a^{+}_{110}+4b^{+}_{000}+8c^{+}_{000}+a^{-}_{101}+a^{-}_{110}+4b^{-}_{000}+8c^{-}_{000})\vspace{0.2cm}\\
&+\dfrac{2}{3}(a^{+}_{101}-a^{-}_{101})z,
\end{array}
$$
Thus the periodic solution that emerge from the cylinder $\textsl{C}$ for system \eqref{system_noncontinuous} converges to the periodic solution with initial condition $(z_{0},1)\in\textsl{C}$ when $\varepsilon$ is sufficiently small, where
$$
z_{0}=-\dfrac{3\pi(a^{+}_{101}+a^{+}_{110}+4b^{+}_{000}+8c^{+}_{000}+a^{-}_{101}+a^{-}_{110}+4b^{-}_{000}+8c^{-}_{000})}{16(a^{+}_{101}-a^{-}_{101})}.
$$

It is easy to check that when $\delta=1$ both expressions of $M_{\delta}$ and $z_{0}$ coincides.
\begin{table}[!h]
            \begin{minipage}[b]{0.45\linewidth}
         \begin{tabular}{|>{\centering\arraybackslash}m{1.5cm} ||>{\centering\arraybackslash}m{1cm}|>{\centering\arraybackslash}m{1cm}|}
           \hline
                   \multicolumn{3}{|c|}{Table 1: Case $m=1$} \\
                   \hline
                   \hline
          \centering \textbf{{\backslashbox {$n$}{$p$}}} & \textbf{0} & \textbf{1}\\
         \hline
         \hline
           \textbf{0} & 0 & 1 \\
         \hline
         \textbf{1}  & 1 & 1\\
         \hline
         \end{tabular}
\end{minipage} \hfill
\begin{minipage}[b]{0.5\linewidth}
\begin{tabular}{|>{\centering\arraybackslash}m{1.5cm} ||>{\centering\arraybackslash}m{1.1cm}|>{\centering\arraybackslash}m{1cm}|>{\centering\arraybackslash}m{1cm}|}
  \hline
          \multicolumn{4}{|c|}{Table 2: Case $m=2$} \\
          \hline
          \hline
 \centering \textbf{{\backslashbox {$n$}{$p$}}} & \textbf{0} & \textbf{1} & \textbf{2}\\
\hline
\hline
  \textbf{0} & 0 (1)* & 1 & 2 \\
\hline
\textbf{1}  & 1 & 1 & 2\\
\hline
\textbf{2}  & 2 & 2 & 2\\
\hline
\end{tabular}
\end{minipage}

\vspace{0.5cm}

\begin{minipage}[b]{0.48\linewidth}
\begin{tabular}{|>{\centering\arraybackslash}m{1.5cm} ||>{\centering\arraybackslash}m{1cm}|>{\centering\arraybackslash}m{1cm}|>{\centering\arraybackslash}m{1cm}|>{\centering\arraybackslash}m{1cm}|}
  \hline
          \multicolumn{5}{|c|}{Table 3: Case $m=3$} \\
          \hline
          \hline
 \centering \textbf{{\backslashbox {$n$}{$p$}}} & \textbf{0} & \textbf{1} & \textbf{2} & \textbf{3}\\
\hline
\hline
  \textbf{0} & 1 (2) & 1 (2) & 2 & 3 \\
\hline
\textbf{1}  & 1 (2) & 1 (2) & 2 & 3\\
\hline
\textbf{2}  & 2 & 2 & 2 & 3\\
\hline
\textbf{3}  & 3 & 3 & 3 & 3\\
\hline
\end{tabular}
\end{minipage}
\vspace{0.4cm}

\caption{Upper bound for the number of limit cycles for particular values of $m$, $n$ and $p$ when $h(x,y)=x/\sqrt{x^2+y^2}$. The number between brackets indicates the upper bound for the non-smooth case, when it is different from the smooth one. The * indicates the case studied previously.}
\end{table}

In short, in the case where $h(x,y)=x/\sqrt{x^2+y^2}$ and the perturbations of system \eqref{system_no_perturbation} are given by \eqref{perturbations2}, we have one limit cycle by considering non-smooth perturbations and no one when we consider smooth ones. This emphasizes the importance of considering non-smooth perturbations. Also, it shows that although Theorems \ref{t1} and \ref{t2} provide the same upper bound for the number of limit cycles by using the Malkin's bifurcation function, the achievement of the number of periodic solutions in each case may be different. Finally, observe that in both cases, the upper bound $s=2=\max\{2,0,0\}$ is not reach.

\smallskip

It is not arduous to exhibit other examples where the number of limit cycles by considering non-smooth perturbations is greater than when we consider smooth ones, but the expressions of $M_{\delta}$ and mainly the zeros $z_{0}$ may become huge and we will not present here. Despite of it, we exhibit some tables indicating the upper bound for the number of limit cycles that can bifurcate from smooth and non-smooth perturbations for the case where $m\geq n,p$ and $a_{ijk}^{\pm}=0$ for all $i,j,k\leq\max\{m,n,p\}-1$, with $m=1,2,3$. This calculations were performed with the help of the algebraic manipulator Wolfram Mathematica.

\smallskip

Finally, we stress out that the same analysis can be performed by considering different expressions of the function $h(x,y)$, i.e., changing the arrangement of the periodic solutions on the cylinder $\textsl{C}$.

\smallskip

Indeed, by considering $h(x,y)=xy/(x^2+y^2)$, we achieve all the necessary suppositions about such function considering the same perturbations and cases of the previous discussion we obtain the following tables.
\begin{table}[!h]
            \begin{minipage}[b]{0.45\linewidth}
         \begin{tabular}{|>{\centering\arraybackslash}m{1.5cm} ||>{\centering\arraybackslash}m{1cm}|>{\centering\arraybackslash}m{1cm}|}
           \hline
                   \multicolumn{3}{|c|}{Table 4: Case $m=1$} \\
                   \hline
                   \hline
          \centering \textbf{{\backslashbox {$n$}{$p$}}} & \textbf{0} & \textbf{1}\\
         \hline
         \hline
           \textbf{0} & 0 & 1 \\
         \hline
         \textbf{1}  & 0 (1) & 1\\
         \hline
         \end{tabular}
\end{minipage} \hfill
\begin{minipage}[b]{0.5\linewidth}
\begin{tabular}{|>{\centering\arraybackslash}m{1.5cm} ||>{\centering\arraybackslash}m{1cm}|>{\centering\arraybackslash}m{1cm}|>{\centering\arraybackslash}m{1cm}|}
  \hline
          \multicolumn{4}{|c|}{Table 5: Case $m=2$} \\
          \hline
          \hline
 \centering \textbf{{\backslashbox {$n$}{$p$}}} & \textbf{0} & \textbf{1} & \textbf{2}\\
\hline
\hline
  \textbf{0} & 1 & 1 & 2 \\
\hline
\textbf{1}  & 1 & 1 & 2\\
\hline
\textbf{2}  & 1 (2) & 1 (2) & 2\\
\hline
\end{tabular}
\end{minipage}

\vspace{0.5cm}

\begin{minipage}[b]{0.5\linewidth}
\begin{tabular}{|>{\centering\arraybackslash}m{1.5cm} ||>{\centering\arraybackslash}m{1cm}|>{\centering\arraybackslash}m{1cm}|>{\centering\arraybackslash}m{1cm}|>{\centering\arraybackslash}m{1cm}|}
  \hline
          \multicolumn{5}{|c|}{Table 6: Case $m=3$} \\
          \hline
          \hline
 \centering \textbf{{\backslashbox {$n$}{$p$}}} & \textbf{0} & \textbf{1} & \textbf{2} & \textbf{3}\\
\hline
\hline
  \textbf{0} & 2 & 2 & 2 & 3 \\
\hline
\textbf{1}  & 2 & 2 & 2 & 3\\
\hline
\textbf{2}  & 2 & 2 & 2 & 3\\
\hline
\textbf{3}  & 2 (3) & 2 (3) & 2 (3) & 3\\
\hline
\end{tabular}
\end{minipage}

\vspace{0.4cm}

\caption{Upper bound for the number of limit cycles for particular values of $m$, $n$ and $p$ when $h(x,y)=xy/(x^2+y^2)$. The upper bound for the number of periodic solutions and the dependence of them in terms of $m$, $n$ and $p$ change according to function $h$.}
\end{table}

Comparing the tables for both expressions of $h(x,y)$ we can see that the bifurcation of periodic orbits depends on the shape of the periodic orbits on $\textsl{C}$. Nevertheless, again the upper bound for the number of periodic orbits when we perform non-smooth perturbations is greater than considering smooth perturbations.

\section*{Acknowledgments}

The first author is partially supported by the grants FP7 PEOPLE-2012-IRSES-318999 and CNPq-Brasil 478230/2013-3. The second author is supported by the FAPESP-BRAZIL grant 2010/18015-6. The third author is supported by the FAPESP-BRAZIL grant 2012/18780-0.

\end{document}